\UseRawInputEncoding
\documentclass{amsart}

\usepackage{graphicx}
\usepackage{xcolor}
\usepackage{amsmath,amsthm,amstext,amssymb,amsfonts,latexsym}

\newtheorem*{thmb}{Theorem B}
\newtheorem*{thmc}{Theorem A}
\newtheorem{thm}{Theorem}[section]

\newtheorem{lem}[thm]{Lemma}
\newtheorem{cor}[thm]{Corollary}
\newtheorem{prop}[thm]{Proposition}

\theoremstyle{definition}
\newtheorem{defi}{Definition}

\newtheorem*{conj}{CPE Conjecture}

\theoremstyle{remark}
\newtheorem{remark}[thm]{Remark}

\numberwithin{equation}{section}

\allowdisplaybreaks[1]

\begin{document}

\title{Critical Metrics of the Volume Functional on three-Dimensional Manifolds }

\author{Huiya He*}
\address{Center of Mathematical Sciences, Zhejiang University, Hangzhou 310027, P. R. China}
\email{hehuiya2019@163.com}

\thanks{}

\date{}
	
\begin{abstract}
In this paper, we prove the three-dimensional $CPE$ conjecture with non-negative Ricci curvature. Moreover, we establish a classification result on three-dimensional vacuum static space with non-negative Ricci curvature. Finally, we show that  a three-dimensional compact, oriented, connected Miao-Tam critical metric with smooth boundary, non-negative Ricci curvature and  non-negative potential function is isometric to a geodesic ball in a simply connected space form $\mathbb{R}^3$ or $\mathbb{S}^3$. 

\end{abstract}
	
\maketitle
	

\section{Introduction}
\label{Sec:Intro}



Let $(M^n, g)$  be a compact, oriented, connected Riemannian manifold with dimension $n$ at least three,  $\mathcal{M}$ be the set of Riemannian metrics on $M^n$ of unitary volume,  $\mathcal{C}\subset\mathcal{M}$ be the set of Riemannian metrics with constant scalar curvature. Define the total scalar curvature functional $\mathcal{R}:\mathcal{M}\rightarrow\mathbb{R}$ as
\begin{eqnarray*}  
\mathcal{R}(g)=\int_{M^n}R_gdM_g,
\end{eqnarray*}
where  $R_g$ is the scalar curvature on $M^n$. 
It is well-known that the formal $L^2$-adjoint of the linearization of the scalar curvature operator ­$\mathfrak{L}_g$ at $g$ is defined as
\begin{eqnarray*} 
\mathfrak{L}^*_g(f):=-(\Delta_g f)g+Hess_g f-fRic_g,
\end{eqnarray*}  
where $f$ is a smooth function on $M^n$, and $\Delta_g$, $Hess_g$ and  $Ric_g$ stand for the Laplacian, the Hessian form and  the Ricci curvature tensor on $M^n$, respectively.  
An Euler-Lagrange equation is given by
\begin{eqnarray}  \label{guanhe10002}
\mathfrak{L}^*_g(f)=Ric_g-\frac{R_g}{n}g.
\end{eqnarray}

As the terminology used in \cite{B1987,CHY2014,QY2013}, a $CPE$ metric can be defined as follows.

\begin{defi}\label{CPE}
A $CPE$ metric is a three-tuple $(M^n,g,f)$$(n\geq 3)$, where $(M^n,g)$ is a compact (without boundary), oriented Riemannian manifold  with constant scalar curvature and $f:M^n\rightarrow\mathbb{R}$ is a non-constant smooth function satisfying equation \eqref{guanhe10002}. Such a function $f$ is called a potential function.
\end{defi}

Besse \cite{B1987} first posed the following conjecture on $CPE$ metrics:

\begin{conj}
A $CPE$ metric is always Einstein.
\end{conj}

In last decades, many researchers  tried to solve this conjecture, but only partial results were proved. 
In 1983, Lafontaine \cite{L1983} showed the $CPE$ conjecture under a locally conformally flat assumption and $Ker\ \mathfrak{L}^*_g(f)\neq 0$.
Chang, Hwang and Yun \cite{CHY2014} avoided the condition on $Ker\ \mathfrak{L}^*_g(f)$.
In 2000, Hwang \cite{H2000} obtained the conjecture provided $f\geq-1$.
In 2013, Qing and Yuan \cite{QY2013} obtained a positive answer for Bach-flat $CPE$ metrics manifolds. 
  For more references on $CPE$ metrics, see  \cite{CHY2012,CHY2016,FY2018,L2015}  and references therein.

When the dimension $n=3$, Baltazar \cite{B2017} proved that
the $CPE$ conjecture is true for three-dimensional manifolds with non-negative sectional curvature.
In this paper, we are able to prove that the $CPE$ conjecture is true for three-dimensional manifolds with non-negative Ricci curvature.

One of our main results is the following:
\begin{thm}\label{thmcpe}
Let $(M^3, g, f)$ be a three-dimensional compact, oriented, connected $CPE$ metric with non-negative Ricci curvature. Then $M^3$  is
isometric to a standard sphere $\mathbb{S}^3$.
\end{thm}

%
%

In this article, we also pay attention to  another kind of critical point equation. 
As in \cite{BR2018,CEM2013, KS2018},  we say that $g$ is a $V$-static metric if there is a smooth function $f$ on $M^{n}$ and a constant $\kappa$ satisfying the $V$-static equation
\begin{eqnarray} \label{vstatic}
\mathfrak{L}^*_g(f)=\kappa g.
 \end{eqnarray}  

On the one hand, when $\kappa= 0$ in \eqref{vstatic}, as in \cite{HE1975, KO1981}, we can have the definition of vacuum static space.

\begin{defi}A vacuum static space is a three-tuple $(M^n,g,f)$$(n\geq 3)$, where $(M^n,g)$ is a compact, oriented, connected Riemannian manifold  with a smooth boundary $\partial M$ and $f:M^n\rightarrow \mathbb{R}$ is a smooth function such that $f^{-1}(0)=\partial M$ satisfying the overdetermined-elliptic system 
\begin{eqnarray} \label{static10001}
\mathfrak{L}^*_g(f)=0.
 \end{eqnarray}  
 Such a function $f$ is called a potential function.

\end{defi}

Kobayashi and Obata \cite{KO1981} (for dimension $n=3$, see \cite{L1980}) proved that  a static metric
$g$ is isometric to a warped-product metric of constant scalar curvature near
the hypersurface $f^{-1}(c)$ for a regular value $c$ provided $g$ is locally conformally flat.
In the paper \cite{FM1975}, Fischer and Marsden showed the local surjectivity for the scalar curvature as a map from the space of
metrics to the space of functions at a non-static metric on a closed manifold and  raised the possibility of identifying all compact vacuum static spaces.

Obata \cite{O1962} showed that the only compact Einstein manifold which satisfies \eqref{static10001} is the standard sphere. 
 Locally conformally flat vacuum static spaces have been completely classified in \cite{K1982} and \cite{L1983} independently.
 When dimension $n=3$, as Ambrozio state in \cite{A2017}, the classification of locally conformally flat vacuum static spaces is as follows.

\begin{thmc}[Kobayashi \cite{K1982}, Lafontaine \cite{L1983}] \label{theoremc}

 Let $(M^{3}, g, f)$ be a vacuum static space with positive scalar curvature. If $(M^{3}, g)$ is locally conformally flat, then $(M^{3}, g, f)$ is covered by a static space that is equivalent to one of the following  tuples:
 
i) The standard round hemisphere
$$
\left(\mathbb{S}_{+}^{3}, g_{\text{can}}, f=x_{n+1}\right),
$$
\indent\indent where $g_{\text {can }}$ is the metric of standard unit round sphere.

ii) The standard cylinder over $\mathbb{S}^{2}$ with the product metric
$$
\left(\left[0, \frac{\pi}{\sqrt{3}}\right] \times \mathbb{S}^{2}, g_{\text {prod}}=d t^{2}+\frac{1}{3} g_{\text {can }}, f=\frac{1}{\sqrt{3}} \sin (\sqrt{3} t)\right).
$$

iii) For some $m \in\left(0, \frac{1}{3 \sqrt{3}}\right),$ the tuple
$$
\left(\left[r_{h}(m), r_{c}(m)\right] \times \mathbb{S}^{2}, g_{m}=\frac{d r^{2}}{1-r^{2}-\frac{2 m}{r}}+r^{2} g_{\text {can }}, f_{m}=\sqrt{1-r^{2}-\frac{2 m}{r}}\right),
$$
\indent\indent where $r_{h}(m)<r_{c}(m)$ are the positive zeroes of $f_{m}$.

\end{thmc}

Based on a similar idea from \cite{CC2012,CCCMM2014}, Qing and Yuan \cite{QY2013} obtained a classification result for Bach-flat vacuum static spaces in any dimension.
Ambrozio \cite{A2017} got some classification results for compact simply connected static three-manifolds with positive scalar curvature.
Baltazar-Ribeiro \cite{BR2018} gave a classification theorem for  compact static three-manifolds with non-negative sectional  curvature.
For more information, see  \cite{HZ2018,L2009,QY2016,Y2015}  and references therein.

In this paper, we are able to prove the following theorem on three-dimensional V-static space with non-negative Ricci curvature.
\begin{thm}\label{vstatic}
Let $(M^3, g, f)$ be a three-dimensional compact, oriented, connected V-static metric with smooth boundary $\partial M$ and non-negative Ricci curvature.
If $f$ and $\kappa$ satisfy one of the following conditions:
 \begin{itemize}
		\item [(i)] $\kappa\geq 0$ and $f>0$;
		\item [(ii)] $\kappa\leq 0$ and $f<0$.	
		 \end{itemize}			
Then $M^3$ is locally conformally flat.
\end{thm}

Specifically, when $\kappa=0$, we can get the following corollary.
\begin{cor}\label{thmstatic}
Let $(M^3, g, f)$ be a three-dimensional compact, oriented, connected vacuum static space with smooth boundary $\partial M$, non-negative Ricci curvature and positive $f$. Then $(M^{3}, g, f)$ is covered by a static space that is equivalent to one of the following tuples:
 
i) The standard round hemisphere
$$
\left(\mathbb{S}_{+}^{3}, g_{\text{can}}, f=x_{n+1}\right),
$$
\indent\indent where $g_{\text {can }}$ is the metric of standard unit round sphere.

ii) The standard cylinder over $\mathbb{S}^{2}$ with the product metric
$$
\left(\left[0, \frac{\pi}{\sqrt{3}}\right] \times \mathbb{S}^{2}, g_{\text {prod}}=d t^{2}+\frac{1}{3} g_{\text {can }}, f=\frac{1}{\sqrt{3}} \sin (\sqrt{3} t)\right).
$$

%

\end{cor}

On the other hand, when $\kappa\neq 0$ in \eqref{vstatic}, in \cite{MT2009}, Miao and Tam have studied the volume functional on the space of constant scalar curvature metrics with a prescribed boundary metric, and derived sufficient or necessary condition for a metric to be a critical point.
As in \cite{KS2018,MT2011}, we can define Miao-Tam critical metric as follows.
\begin{defi}A Miao-Tam critical metric is a three-tuple $(M^n,g,f)$$(n\geq 3)$, where $(M^n,g)$ is a compact, oriented, connected Riemannian manifold  with a smooth boundary $\partial M$ and $f:M^n\rightarrow \mathbb{R}$ is a smooth function such that $f^{-1}(0)=\partial M$ satisfying the overdetermined-elliptic system 
\begin{eqnarray} \label{guanhe10001}
\mathfrak{L}^*_g(f)=g.
 \end{eqnarray}  
 Such a function $f$ is called a potential function.

\end{defi}

In  \cite{MT2011}, Miao-Tam studied these critical metrics under Einstein condition. More precisely, they obtained the following result.

\begin{thmb}[Miao-Tam \cite{MT2011}]\label{theoremb}
Let $(M^n, g, f)$ be a connected, compact Einstein Miao-Tam critical metric with smooth boundary $\partial M$, then $M^n$ is isometric to a geodesic ball in a simply connected space form $\mathbb{R}^n$, $\mathbb{H}^n$ or $\mathbb{S}^n$.
\end{thmb}
In the same article they obtained the same result under the assumption that the metric is locally conformally flat instead of Einstein. 
Based on a work of Cao-Chen \cite{CC2013}, Barros-Diogenes-Ribeiro \cite{BDR2015} showed that a 4-dimensional Bach-flat simply connected, compact Miao-Tam critical metric with boundary isometric to a standard sphere $\mathbb{S}^3$ must be isometric to a geodesic ball in a simply connected space form $\mathbb{R}^4$, $\mathbb{H}^4$ or $\mathbb{S}^4$.


It is important to note that Baltazar-Ribeiro \cite{BR2018} provided a general Bochner type formula to show that a three-dimensional compact, oriented, connected Miao-Tam critical metric $(M^3, g, f)$ with smooth boundary $\partial M$ and non-negative sectional curvature, with $f$ assumed to be non-negative, is isometric to a geodesic ball in a simply connected space form $\mathbb{R}^3$ or $\mathbb{S}^3$. 
We refer to \cite{ Bd2017,BDRR2017,BR2017,SW2019} for more related results.




Similar to our theorem on $CPE$ metric, we can establish the following result.
\begin{thm}\label{thmmiao}
Let $(M^3, g, f)$ be a three-dimensional compact, oriented, connected Miao-Tam critical metric with smooth boundary $\partial M$, non-negative Ricci curvature and non-negative $f$. Then $M^3$ is isometric to a geodesic ball in a simply connected space form $\mathbb{R}^3$ or $\mathbb{S}^3$.
\end{thm}

\begin{remark}
As a direct corollary of Theorem \ref{vstatic},  a classification result of Miao-Tim critical metric under the condition of $f>0$ can be deduced. In Theorem \ref{thmmiao}, we optimize this condition to  $f\geq 0$. 

\end{remark}

%
%
%
%
%
%
%
%
%

The organization of the paper is as follows. In section 2, we unify the form of  equation of $CPE$ metric, Miao-Tam critical metric and vacuum static space. Moreover,  we
state four lemmas as the preparation. In section 3, we give the proof of our main theorems.

{Acknowledgement:} The author is very grateful to Professor Haizhong Li and Professor Hongwei Xu for their guidance and constant support.


\section{Preliminaries}

Let $(M^n,g)(n\geq3)$ be an $n$-dimensional  compact, orientable Riemannian manifold. 
In what follows, we adopt, without further comment, the moving frame notation.
For any  $p\in M^n$, we choose $e_1,\ldots,e_n$ as a local orthonormal frame field at $p$, $\omega_1,\ldots,\omega_n$ as its dual coframe field,  $g_{ij}=\delta_{ij}$. 
Here and hereafter the Einstein convention of summing over the repeated indices will be adopted. 

The
decomposition of the Riemannian curvature tensor  into irreducible components yields
\begin{eqnarray}\label{00003}
R_{ijkl}&=&W_{ijkl}+\frac{1}{n-2}(R_{ik}\delta_{jl}-R_{il}\delta_{jk}+R_{jl}\delta_{ik}-R_{jk}\delta_{il})\nonumber\\
&&-\frac{R}{(n-1)(n-2)}(\delta_{ik}\delta_{jl}-\delta_{il}\delta_{jk}),
\end{eqnarray}
where $R_{ijkl}$ are the components of the Riemannian curvature tensor, $W_{ijkl}$ are the components of the Weyl tensor, $  R_{ij}:=\sum_{kl}R_{ikjl}g_{kl}$  are components of Ricci curvature tensor and $R:=\sum_{ij}R_{ij}g_{ij}$ is the scalar curvature of $M^n.$



Let $\phi=\sum_{i,j}\phi_{ij}\omega_i  \otimes      \omega_j$ be a symmetric (0,2)-type tensor defined on $M^n.$   By letting  $\phi_{ij,k}:=\nabla_k\phi_{ij},$ $\phi_{ij,kl}:=\nabla_l\nabla_k\phi_{ij},$ where $\nabla$ is the operator of covariant differentiation on $M^n,$
 we have the following Ricci identities
\begin{eqnarray}
\phi_{ij,kl}-\phi_{ij,lk}=\phi_{mj}R_{mikl}+\phi_{im}R_{mjkl}.\label{11028}
  \end{eqnarray}

The norm of a $(0,4)$-type tensor $T$ is defined as $$|T|^2=|T_{ijkl}|^2=T_{ijkl}T_{ijkl}.$$

By the second Bianchi identity
\begin{equation*}
R_{imkl,j}+R_{imlj,k}+R_{imjk,l}=0,
  \end{equation*}
we have
\begin{equation}\label{11002}
R_{ij,k}-R_{ik,j}=R_{likj,l}
  \end{equation}
and
\begin{equation}\label{11003}
R_{ik,i}=\frac{1}{2}R_{,k}.
  \end{equation}

A Cotton tensor $C_{ijk}$ is given by
\begin{eqnarray}\label{11005}
C_{ijk}=R_{jk,i}-R_{ik,j}-\frac{1}{2(n-1)}(R_{,i}\delta_{jk}-R_{,j}\delta_{ik}).
  \end{eqnarray}


From Definition \ref{CPE} we know that a $CPE$ metric $(M^n, g, \tilde{f})$ satisfies
\begin{eqnarray} \label{00000000001}
-(\Delta\tilde{f})\delta_{ij}+\tilde{f}_{,ij}-\tilde{f}R_{ij}=Ric-\frac{R}{n}\delta_{ij},
 \end{eqnarray}  
where $R$ is a constant.
Replacing $\tilde{f}$ by $f-1$, we can rewrite \eqref{00000000001}  as
\begin{eqnarray} \label{00004}
-(\Delta f)\delta_{ij}+f_{,ij}-fR_{ij}=-\frac{R}{n}\delta_{ij},
 \end{eqnarray}  
 where $-\frac{R}{n}$ is a constant. Obviously, \eqref{00004} is  a V-static equation. Then we can unify the equation of $CPE$ metric, Miao-Tam critical metric and vacuum static space  into the following form:
\begin{eqnarray} \label{00005}
-(\Delta f)\delta_{ij}+f_{,ij}-fR_{ij}=\kappa\delta_{ij},
 \end{eqnarray}  
where $\kappa$ is a constant. By tracing the above formula, we can get
 \begin{eqnarray} \label{000066}
\Delta f=\frac{fR+n\kappa}{1-n}.
 \end{eqnarray}  
Substituting \eqref{000066} into \eqref{00005}, we have
\begin{eqnarray} \label{000077}
fR_{ij}-f_{,ij}-\frac{fR}{n-1}\delta_{ij}=\frac{\kappa}{n-1}\delta_{ij}.
 \end{eqnarray}

Now we present four lemmas as the preparation  for proving our main theorems. 
\begin{lem}\label{0010000}
Let $(M^n,g)(n\geq3)$ be an $n$-dimensional Riemannian manifold, $f$ be a smooth solution to \eqref{00005}.     Then $g$ has constant scalar curvature.
 \end{lem}
\begin{proof}
As the proof of Theorem 7 in Miao-Tam \cite{MT2009}, differenting \eqref{00005}, we get
\begin{eqnarray} 
-f_{,kkj}\delta_{ij}+f_{,ijj}-f_{,j}R_{ij}-fR_{ij,j}=0.
 \end{eqnarray}  
From Ricci identity and \eqref{11003}, we can derive that scalar curvature $R$ is constant.
Hence we complete the proof of Lemma \ref{0010000}.
\end{proof}

\begin{lem}\label{10000}
Let $(M^n,g)(n\geq3)$ be an $n$-dimensional Riemannian manifold, $f$ be a smooth solution to \eqref{00005}.     Then
\begin{eqnarray} 
fC_{kij}=  f_{,h}R_{hjik}+\frac{R}{n-1}(f_{,k}\delta_{ij}-f_{,i}\delta_{kj})-(   f_{,k}R_{ij} -f_{,i}R_{kj}        ).
 \end{eqnarray}  \end{lem}
 
 \begin{remark}
 As in \cite{BDR2015} and \cite{BR2017},  we can derive 
\begin{eqnarray}
fC_{kij}=T_{kij}+W_{kijs}f_{,s}\label{00007}
  \end{eqnarray}
from \eqref{00003} and Lemma \ref{10000},  where $T_{kij}$ is defined as
\begin{eqnarray*}
T_{kij}&=&\frac{n-1}{n-2}(R_{jk}f_{,i}-R_{ij}f_{,k})-\frac{R}{n-2}(\delta_{jk}f_{,i}-\delta_{ij}f_{,k})\\
&&+\frac{1}{n-2}(\delta_{jk} R_{is}f_{,s}-\delta_{ij}R_{ks}f_{,s}).
  \end{eqnarray*}
\end{remark}

\begin{proof}
Differenting \eqref{000077}, we get
\begin{eqnarray} 
fR_{ij,k}=f_{,ijk}+\frac{Rf_{,k}}{n-1}\delta_{ij}-f_{,k}R_{ij}.\label{00002}
 \end{eqnarray}  

Combining this equation with Ricci identity, we have
\begin{eqnarray*} 
fC_{kij}&=&f(R_{ij,k}-R_{kj,i})\\
&=&f_{,ijk}-f_{,kji} +\frac{R}{n-1}(f_{,k}\delta_{ij}-f_{,i}\delta_{kj})-(   f_{,k}R_{ij} -f_{,i}R_{kj}        )\\
&=&  f_{,h}R_{hjik}+\frac{R}{n-1}(f_{,k}\delta_{ij}-f_{,i}\delta_{kj})-(   f_{,k}R_{ij} -f_{,i}R_{kj}        ).
 \end{eqnarray*}  
Hence we complete the proof of Lemma \ref{10000}.
\end{proof}

\begin{lem}\label{guanhelem00002}
Let $(M^n,g)(n\geq3)$ be an $n$-dimensional Riemannian manifold, $f$ be a smooth solution to \eqref{00005}.     Then
\begin{eqnarray} 
fR_{ij}C_{kij,k}&=&fR_{ij}R_{jk}R_{ki}-fR_{ij}R_{hk}R_{ikjh}-\frac{1}{2}<\nabla f, \nabla |Ric|^2>\nonumber\\
&&-2R_{ij}f_{,k}C_{kij}+\frac{n\kappa}{n-1}|\mathring{Ric}|^2.
 \end{eqnarray}  
\end{lem}

\begin{proof}

According to Lemma \ref{10000}, we have 
\begin{eqnarray*} 
fR_{ij}C_{kij,k}&=&R_{ij}(fC_{kij})_{,k}-R_{ij}f_{,k}C_{kij}\\
&=&R_{ij}\left[                  f_{,h}R_{hjik}+\frac{R}{n-1}(f_{,k}\delta_{ij}-f_{,i}\delta_{kj})-(   f_{,k}R_{ij} -f_{,i}R_{kj}        )              \right]_{,k}\\
&&-R_{ij}f_{,k}C_{kij}\\
&=&R_{ij}                  f_{,hk}R_{hjik}+R_{ij}    f_{,h}R_{hjik,k}+            \frac{R^2}{n-1}\Delta f    -    \frac{R}{n-1} f_{,ij}R_{ij} \\
&& -            \Delta f|Ric|^2 -R_{ij}f_{,k}R_{ij,k} +R_{ij} f_{,ik}R_{kj}                 -R_{ij}f_{,k}C_{kij}.
 \end{eqnarray*}  

From \eqref{11002}, \eqref{11005} and the fact  $R$ is a constant, we know that $ R_{hjik,k}=C_{jhi}$. Substituting this equation,  \eqref{000066} and  \eqref{000077} into  equation above, we have
\begin{eqnarray*} 
fR_{ij}C_{kij,k}&=&R_{ij}            (    fR_{hk}-\frac{fR}{n-1}\delta_{hk}-\frac{\kappa}{n-1}\delta_{hk}   )R_{hjik}+R_{ij}    f_{,h}  C_{jhi}  \\
&&   +            \frac{R^2}{n-1}\frac{fR+n\kappa}{1-n}   -    \frac{R}{n-1} (   fR_{ij}-\frac{fR}{n-1}\delta_{ij}-\frac{\kappa}{n-1}\delta_{ij}      )R_{ij} \\
&& -           \frac{fR+n\kappa}{1-n}|Ric|^2 -R_{ij}f_{,k}R_{ij,k} +R_{ij} (    fR_{ik}-\frac{fR}{n-1}\delta_{ik}\\
&&-\frac{\kappa}{n-1}\delta_{ik}         )R_{kj}                 -R_{ij}f_{,k}C_{kij}.
 \end{eqnarray*}  
That is 
\begin{eqnarray*} 
fR_{ij}C_{kij,k}&=&fR_{ij}R_{jk}R_{ki}-fR_{ij}R_{hk}R_{ikjh}-\frac{1}{2}<\nabla f, \nabla |Ric|^2>\\
&&-2R_{ij}f_{,k}C_{kij}+\frac{n\kappa}{n-1}|Ric|^2-\frac{R^2\kappa}{n-1}.
 \end{eqnarray*}

Thus we complete the proof of Lemma \ref{guanhelem00002}.
\end{proof}


\begin{lem}\label{guanhelem00003}
Let $(M^3,g)$ be a three-dimensional  Riemannian manifold with non-negative Ricci curvature. Then
\begin{eqnarray*}
{6{R_{ij}}{R_{hj}}{R_{hi}} - 5R{{\left| {Ric} \right|}^2} + {R^3}}\geq 0.
\end{eqnarray*}

\end{lem}

\begin{proof}

For any fixed point $p\in M$, we choose an orthonormal basis $\{e_i\}_{i=1}^3$ in $T_p M$ such that
$$R_{ij}=\rho_i\delta_{ij},\ \ \ R=\rho_1+\rho_2+\rho_3.$$
Without loss of generality, we can assume $0\leq \rho_1\leq \rho_2 \leq \rho_3$. 

Similar to the proof of Lemma 2.2 in \cite{HL2008}, we can compute that
\begin{eqnarray*}
{6{R_{ij}}{R_{hj}}{R_{hi}} - 5R{{\left| {Ric} \right|}^2} + {R^3}}&=&6\sum_{i=1}^3\rho_i^3 - 5\sum_{i=1}^3\rho_i\sum_{j=1}^3\rho_j^2 +(\sum_{i=1}^3\rho_i)^3   \\
&=&9\sum_{i=1}^3\rho_i^3 - 6\sum_{i=1}^3\rho_i\sum_{j=1}^3\rho_j^2 +(\sum_{i=1}^3\rho_i)^3\\
&&-\sum_{i=1}^3\left(3\rho_i^3 - \rho_i\sum_{j=1}^3\rho_j^2 \right)\\
&=&\sum_{i,j,k=1}^3\rho_i(\rho_i-\rho_j)(\rho_i-\rho_k)-\sum_{i,j=1}^3\rho_i(\rho_i-\rho_j)^2\\
&=&\sum_i\sum_{j\neq k}\rho_i(\rho_i-\rho_j)(\rho_i-\rho_k)\\
&=&\sum_{i,j,k\neq}\rho_i(\rho_i-\rho_j)(\rho_i-\rho_k)\\ 
&=&2[\rho_1(\rho_1-\rho_2)(\rho_1-\rho_3)+\rho_2(\rho_2-\rho_1)(\rho_2-\rho_3)\\
&&+\rho_3(\rho_3-\rho_1)(\rho_3-\rho_2)]\\ 
&=&2[\rho_1(\rho_1-\rho_2)(\rho_1-\rho_3)\\
&&+ (\rho_2-\rho_3)^2(\rho_2+\rho_3-\rho_1)]\geq 0.
\end{eqnarray*}
Thus we complete the proof of Lemma \ref{guanhelem00003}.
\end{proof}

\section{Proof of main theorems}

In this section, we prove the main theorems of our paper. Let $(M^n,g)(n\geq 3)$ be an $n$-dimensional compact, oriented, connected Riemannian manifold, $f$ be a smooth solution to \eqref{00005}. From a direct computation and the definition of $C_{ijk},$ we get
\begin{eqnarray*}
 div\left( {f^2\nabla {{\left| {Ric} \right|}^2}} \right) 
   &=&2f<\nabla f, \nabla |Ric|^2>+f^2\Delta |Ric|^2\\
 &=&2f<\nabla f, \nabla |Ric|^2>+2f^2|\nabla Ric|^2 +2f^2R_{ij}R_{ij,kk}  \\
 &=&2f<\nabla f, \nabla |Ric|^2>+2f^2|\nabla Ric|^2 +2f^2R_{ij}(C_{kij}+R_{jk,i})_{,k}  \\
   &=&2f<\nabla f, \nabla |Ric|^2>+2f^2|\nabla Ric|^2 +2f^2R_{ij}C_{kij,k}\\
   && +2f^2R_{ij}R_{jk,ik}.   
     \end{eqnarray*}
     Combining  equation above with Lemma \ref{guanhelem00002} and Ricci identity, we have   
 \begin{eqnarray}\label{0000000031}
 div\left( {f^2\nabla {{\left| {Ric} \right|}^2}} \right) 
    &=&4f^2R_{ij}R_{jk}R_{ki}-4f^2R_{ij}R_{hk}R_{ikjh}-8fR_{ij}f_{,k}C_{kij}\nonumber\\
    &&+\frac{4n\kappa f}{n-1}|\mathring{Ric}|^2-4f^2R_{ij}C_{kij,k}+2f^2|\nabla Ric|^2 \nonumber\\
    &&+2f^2R_{ij}C_{kij,k} +2f^2R_{ij}R_{jk,ik}\nonumber\\
      &=&6f^2R_{ij}R_{jk}R_{ki}-6f^2R_{ij}R_{hk}R_{ikjh}-8fR_{ij}f_{,k}C_{kij}\nonumber\\
      &&+\frac{4n\kappa f}{n-1}|\mathring{Ric}|^2-2f^2R_{ij}C_{kij,k}+2f^2|\nabla Ric|^2. 
   \end{eqnarray}

Since 
\begin{eqnarray*}
(f^2R_{ij}C_{kij})_{,k}=2ff_{,k}R_{ij}C_{kij}+f^2R_{ij,k}C_{kij}+f^2R_{ij}C_{kij,k},
\end{eqnarray*}
it follows that
  \begin{eqnarray*}
 div\left( {f^2\nabla {{\left| {Ric} \right|}^2}} \right) 
     &=&6f^2R_{ij}R_{jk}R_{ki}-6f^2R_{ij}R_{hk}R_{ikjh}-4fR_{ij}f_{,k}C_{kij}\\
     &&+\frac{4n\kappa f}{n-1}|\mathring{Ric}|^2+2f^2R_{ij,k}C_{kij}- 2(f^2R_{ij}C_{kij})_{,k}                                      \\
     &&            +2f^2|\nabla Ric|^2\\
      &=&6f^2R_{ij}R_{jk}R_{ki}-6f^2R_{ij}R_{hk}R_{ikjh}-4fR_{ij}f_{,k}C_{kij}\\
      &&+\frac{4n\kappa f}{n-1}|\mathring{Ric}|^2+f^2|C|^2- 2(f^2R_{ij}C_{kij})_{,k}                                                  +2f^2|\nabla Ric|^2. 
     \end{eqnarray*}

According to \eqref{00007} and the fact that $W_{ijkl}=0$ when $n=3,$ we have
\begin{eqnarray*}
fC_{kij}=T_{kij},
\end{eqnarray*}
then from the definition of $T_{kij},$ we can get
\begin{eqnarray}\label{000000031}
f|C|^2=T_{kij}C_{kij}=-4C_{kij}R_{ij}f_{,k}.
\end{eqnarray}
Besides, from \eqref{00003} we know that
 \begin{eqnarray}\label{000000032}
 R_{ij}R_{jk}R_{ki}-R_{ijkl}R_{ik}R_{jl}=3R_{ij}R_{jk}R_{ki}-\frac{5R}{2}|Ric|^2+\frac{R^3}{2}.
 \end{eqnarray}

Thus we can conclude that when $n=3$, 
  \begin{eqnarray}\label{000031}
 div\left( {f^2\nabla {{\left| {Ric} \right|}^2}} \right) 
          &=&    3f^2(          6R_{ij}R_{jk}R_{ki}-5R|Ric|^2+R^3)      +6\kappa f|\mathring{Ric}|^2+2f^2|C|^2\nonumber\\
          &&- 2(f^2R_{ij}C_{kij})_{,k}                                                  +2f^2|\nabla Ric|^2. 
     \end{eqnarray}

To prove Theorem \ref{thmcpe}, we need the following results by Hwang \cite{H2000} and Cheng \cite{C2001}.
\begin{prop}[\cite{H2000}]\label{prop}
Let $(M^n,g,f)$ be a $CPE$ metric with $f$ non-constant. Then the set $\{ x\in M^n:f(x)=-1       \}$ has measure zero.

\end{prop}

\begin{prop}[\cite{C2001} ]\label{prop2}
Let $M^{n}$ be a compact connected oriented locally conformally flat $n$-dimensional Riemannian manifold with constant scalar curvature. If the Ricci curvature of $M^{n}$ is non-negative, then $M^{n}$ is isometric to a space form or a Riemannian product $\mathbb{S}^{n-1}(c) \times \mathbb{S}^{1}$, where $c$ is the sectional curvature of $\mathbb{S}^{n-1}$.

\end{prop}

\begin{proof}[Proof of Theorem \ref{thmcpe}]

Let $(M^3, g, \tilde{f})$ be a three-dimensional compact, oriented, connected $CPE$ metric with non-negative Ricci curvature. Then $\tilde{f}$ is a non-constant solution to \eqref{00000000001}.
Hence  $f=1+\tilde{f}$ satisfies \eqref{00005} with $\kappa=-\frac{R}{3}$. From Lemma \ref{guanhelem00003} and \eqref{000031} we have 
  \begin{eqnarray*}
 div\left( {(1+\tilde{f})^2\nabla {{\left| {Ric} \right|}^2}} \right) 
        &\geq& -2R(1+\tilde{f})|\mathring{Ric}|^2+2(1+\tilde{f})^2|C|^2\\
     &&   - 2((1+\tilde{f})^2R_{ij}C_{kij})_{,k}              +2(1+\tilde{f})^2|\nabla Ric|^2. 
     \end{eqnarray*}

According to a direct computation on \eqref{00000000001}, we can see that $(1+\tilde{f})|\mathring{Ric}|^2=\tilde{f}_{,ij}\mathring{R}_{ij}$, and  from \eqref{11003}  and the fact that $M^3$ has no boundary it is easy to know that   $\int_{M^3}(1+\tilde{f})|\mathring{Ric}|^2 dM_g=0.$  Integraling the  inequality above, we can derive that
  \begin{eqnarray*}
 0     \geq \int_{M^3}2(1+\tilde{f})^2|C|^2  dM_g           +2\int_{M^3}(1+\tilde{f})^2|\nabla Ric|^2 dM_g. 
     \end{eqnarray*}
From Proposition \ref{prop}, we can see that $M^3$ is a locally conformally flat manifold and $\nabla Ric=0$.

Taking trace of \eqref{00000000001}, we have 
  \begin{eqnarray}\label{deltaf}
\Delta \tilde{f}=-\frac{R}{2}\tilde{f}.
     \end{eqnarray}
Multiply both sides of this equation by $\tilde{f}$ and integrate on $M^3$, then we can get
  \begin{eqnarray*}
-\int_{M^3}|\nabla \tilde{f}|^2 dM_g=\int_{M^3}(\Delta \tilde{f})\cdot \tilde{f} dM_g=-\frac{R}{2}\int_{M^3}\tilde{f}^2 dM_g.
     \end{eqnarray*}
According to the fact that $\tilde{f}$ is not a constant, we can conclude that the scalar curvature $R$ is a positive constant. 

By using Proposition \ref{prop2}, we can conclude that $M^3$ is isometric to a standard sphere $\mathbb{S}^3$ or a Riemannian  
product $\mathbb{S}^{2}(c) \times \mathbb{S}^1$, where $c$ is the sectional curvature of $\mathbb{S}^{2}$.

Here if $M^3$ is isometric to $\mathbb{S}^{2}(c) \times \mathbb{S}^1$, we have 
\[(R_{ij})_{3\times 3}=\left( {\begin{array}{*{20}{c}}
  { c   }&{}&{} \\
  {}&{c}&{} \\
  {}&{}&{0}
\end{array}} \right)_{3\times 3,}\]
and $R=2c$.
Substituting \eqref{deltaf} into \eqref{00000000001}, we can deduce that
  \begin{eqnarray*}
\tilde{f}(c\delta_{ij}-R_{ij})+\tilde{f}_{,ij}=R_{ij}-\frac{2c}{3}\delta_{ij}.
     \end{eqnarray*}
Then $\tilde{f}_{,11}=\frac{c}{3}$, $\tilde{f}_{,22}=\frac{c}{3}$, $\tilde{f}_{,33}=-c\tilde{f}-\frac{2c}{3}$. Thus for any $q\in \mathbb{S}^1$, we can get that $\int_{\mathbb{S}^2\times  \{q\}} \Delta_{\mathbb{S}^2\times  \{q\}} \tilde{f} d\text{vol} =\int_{\mathbb{S}^2\times  \{q\}}\frac{2c}{3} d\text{vol}
\neq0$, 
which is a controdiction.

Thus we finish the proof of Theorem \ref{thmcpe}.
\end{proof}

\begin{remark}
In the process  of proving Theorem \ref{thmcpe}, after getting the result that $M^3$ is a locally conformally flat manifold, we can also deduce that $M^3$  is
isometric to a standard sphere $\mathbb{S}^3$
 by using
Corollary 1.3 in \cite{CHY2014}.
\end{remark}

Similarly,  we can prove a three-dimensional V-static metric is locally conformally flat under some suitable conditions.
\begin{proof}[Proof of Theorem \ref{vstatic}]
Let $(M^3, g, f)$ be a three-dimensional compact, oriented, connected V-static metric with smooth boundary $\partial M$ and non-negative Ricci curvature.
 Integraling \eqref{000031}, we can deduce 
 \begin{eqnarray}\label{1010101010101}
 0     \geq             6\kappa \int_{M^3}f|\mathring{Ric}|^2     dM_g    +              2\int_{M^3}f^2|C|^2       dM_g      +2\int_{M^3}f^2|\nabla Ric|^2 dM_g
     \end{eqnarray}
 from Lemma \ref{guanhelem00003}.    Then if $\kappa\geq 0$, $f>0$ or $\kappa\leq 0$, $f<0$,	 $M^3$ is locally conformally flat.
Thus we complete the proof of Theorem \ref{vstatic}.
\end{proof}

When $\kappa=0$, Corollary \ref{thmstatic} can be directly deduced from Theorem \ref{vstatic} and Theorem A.

When $\kappa=1$, from \eqref{0000000031}, we can prove Theorem \ref{thmmiao}  on the basis of the following calculation
\begin{eqnarray*} 
div(f\nabla|Ric|^2)&=&  \frac{1}{f} div(f^2\nabla|Ric|^2)  -<\nabla f, \nabla|Ric|^2> \\
&=& 6fR_{ij}R_{jk}R_{ki}-6fR_{ij}R_{hk}R_{ikjh}-8R_{ij}f_{,k}C_{kij}\nonumber\\
      &&+\frac{4n\kappa }{n-1}|\mathring{Ric}|^2-2fR_{ij}C_{kij,k}+2f|\nabla Ric|^2 -<\nabla f, \nabla|Ric|^2>. 
 \end{eqnarray*}

Using  Lemma \ref{guanhelem00002} again, we have
\begin{eqnarray}\label{00000000000000}
 div\left( {f\nabla {{\left| {Ric} \right|}^2}} \right) 
     &=& 4 fR_{ij}R_{jk}R_{ki}-4fR_{ij}R_{hk}R_{ikjh}\nonumber\\
   &&-4R_{ij}f_{,k}C_{kij}+\frac{2n\kappa}{n-1}|\mathring{Ric}|^2+2f|\nabla Ric|^2.
\end{eqnarray}

\begin{proof}[Proof of Theorem \ref{thmmiao}]
Let $(M^3, g, f)$ be a three-dimensional compact, oriented, connected Miao-Tam critical metric with smooth boundary $\partial M$, non-negative Ricci curvature and non-negative $f$.
Then $f$ is a non-constant solution to \eqref{guanhe10001}.
That is $f$ satisfies \eqref{00005} with $\kappa=1$.
 
\begin{eqnarray*}
 div\left( {f\nabla {{\left| {Ric} \right|}^2}} \right) 
     &=& 4 fR_{ij}R_{jk}R_{ki}-4fR_{ij}R_{hk}R_{ikjh}\nonumber\\
   &&-4R_{ij}f_{,k}C_{kij}+3|\mathring{Ric}|^2+2f|\nabla Ric|^2.
\end{eqnarray*}

As  the proof of Theorem \ref{thmcpe}, integraling this equation, we can derive that
\begin{eqnarray*}
0\geq
  \int_{M^3}f(2|\nabla Ric|^2+ |C|^2) dM_g+\int_{M^3}3|\mathring{Ric}|^2 dM_g
\end{eqnarray*}
from \eqref{000000031}, \eqref{000000032} and  Lemma \ref{guanhelem00003}.
 Hence $M^3$ is an Einstein  manifold. Thus from Theorem B we can complete the proof of Theorem \ref{thmmiao}.
\end{proof}

\begin{remark}
 From the proof of Theorem    \ref{thmmiao}    we can deduce that,  when $\kappa>0$, a V-static metric is Einstein under the condition  of  $Ric\geq 0$ and $f\geq 0$.

\end{remark}

%
%
%
%
%
%
%
%
%
%
%
%
%
%
%
%
%
%



\begin{thebibliography}{99}

\bibitem{A2017}
L. Ambrozio,
\emph{On static three-manifolds with positive scalar curvature. }
J. Differential Geom. \textbf{107} (2017), no. 1, 1-45. 





\bibitem{B1987}
 A.L.  Besse, \emph{Einstein Manifolds.} Springer-Verlay, Berlin, 1987.

%
\bibitem{B2017}
 H.  Baltazar, \emph{On critical point equation of compact manifolds with zero radial weyl curvature.}  Geom. Dedic. \textbf{202} (2019), 337-355.
%

%
%

\bibitem{Bd2017}
A. Barros  and A. da Silva, \emph{Rigidity for critical metrics of the volume functional.} Math. Nachr. \textbf{292} (2019), no. 4, 709-719.




\bibitem{BDR2015}
 A. Barros, R. Di\'{o}genes and E. Ribeiro Jr., \emph{Bach-Flat Critical Metrics of the Volume Functional on 4-Dimensional Manifolds with Boundary.} J. Geom. Anal. \textbf{25} (2015), 2698-2715.




\bibitem{BDRR2017}
 R. Batista,   R. Di\'{o}genes, M. Ranieri  and E. Ribeiro Jr., \emph{Critical metrics of the volum functional on compact three-manifolds with smooth boundary.} J. Geom. Anal. \textbf{27} (2017), 1530-1547.



%
%
%


\bibitem{BR2017}
 H. Baltazar and E. Ribeiro Jr., \emph{Critical metrics of the volume functional on manifolds with boundary.} Proc.  Amer. Math. Soc. \textbf{145} (2017), 3513-3523.



\bibitem{BR2018}
 H. Baltazar  and E.  Ribeiro Jr., \emph{Remarks on critical metrics of the scalar curvature and volume functionals on compact manifollds with boundary.} Pacific J. Math. \textbf{297(1)}  (2018), 29-45.



\bibitem{C2001}
Q. M. Cheng, \emph{Compact locally conformally flat Riemannian manifolds.} Bull. London Math. Soc. 
\textbf{33} (2001), no. 4, 459-465. 



\bibitem{CC2012}
 H.-D. Cao and Q. Chen, 
 \emph{On locally conformally flat gradient steady solitons.} 
 Trans. Amer. Math. Soc. \textbf{364} (2012), 2377-2391.


\bibitem{CC2013}
 H.-D. Cao  and Q.  Chen, \emph{On Bach-flat gradient shrinking Ricci solitions.} Duke Math. J. \textbf{162} (2013), 1149-1169.

\bibitem{CCCMM2014}
 H.-D. Cao, G. Catino, Q. Chen, C. Mantegazza and L. Mazzieri, 
 \emph{Bach-flat gradient
steady Ricci solitons.} Calc. Var. PDE. \textbf{49} (2014), 125-138.


%
\bibitem{CEM2013}
 J. Corvino,  M. Eichmair,  and P. Miao,  \emph{Deformation of scalar curvature and volume.} Math. Ann. \textbf{357} (2013), 551-584.



\bibitem{CHY2012}
J. Chang, S. Hwang  and G. Yun, \emph{Critical point metrics of the total scalar curvature.} Bull. Korean Math. Soc. \textbf{49} (2012), 655-667.



\bibitem{CHY2014}
 J. Chang, S. Hwang and G. Yun, \emph{Total scalar curvature and harmonic curvature.} Taiwanese J. Math. \textbf{18} (2014), 1439-1458.



\bibitem{CHY2016}
 J. Chang, S. Hwang  and G. Yun, \emph{Erratum to: Total scalar curvature and harmonic curvature.} Taiwanese J. Math. \textbf{20(3)} (2016), 699-703.



\bibitem{FM1975}
 A. Fischer and J. Marsden, \emph{Deformations of the scalar curvature.} Duke Math. J. \textbf{42} (1975), 519-547.


\bibitem{FY2018}
Y.  Fang and W.  Yuan, \emph{Brown-York mass and positive scalar curvature II-Besse’s conjecture and related problems.}  Ann. Global Anal. Geom. \textbf{56} (2019), no. 1, 1-15.




\bibitem{H2000}
 S. Hwang, \emph{Critical points of the total scalar curvature functional on the space of metrics of constant scalar curvature.} Manuscripta Math. \textbf{103} (2000), 135-142.

\bibitem{HE1975}
 S. Hawking and G. Ellis, \emph{The Large Scale Structure of Space-Time.} Cambridge University Press, 1975.




\bibitem{HL2008}
 Z.J. Hu, H. Li and Udo Simon, \textit{Schouten curvature functions on locally conformally flat Riemannian manifolds.}  J. Geom. \textbf{88} (2008), 75-100.



%
%
%
%

\bibitem{HZ2018}
G. Huang and F. Zeng,
\emph{The classification of static spaces and related problems. }
Colloq. Math. \textbf{151} (2018), no. 2, 189-202. 











\bibitem{K1982}
 O. Kobayashi,   \emph{A differential equation arising from scalar curvature function.} J. Math. Soc. Japan. \textbf{34} (1982), no. 4, 665-675.



\bibitem{KO1981}
O. Kobayashi and M. Obata, \emph{Conformally-flatness and static space-time, in: Manifolds and Lie Groups.}  Progr. Math. \textbf{14} (1981), Birkh$\ddot{a}$user, 197-206.



\bibitem{KS2018}
J. Kim and J. Shin,  \emph{Four dimensional static and related critical spaces with harmonic curvature.} Pacific J. Math. \textbf{295(2)} (2018), 429-462.






\bibitem{L1980}
 L. Lindblom, \emph{Some properties of static general relativistic stellar models.} J. Math.
Phys. \textbf{21} (1980), 1455-1459.



\bibitem{L1983}
 J. Lafontaine, \emph{Sur la g$\acute{e}$om$\acute{e}$trie d’une g$\acute{e}$n$\acute{e}$ralisation de l'$\acute{e}$quation diff$\acute{e}$rentielle d'Obata.} J. Math. Pures Appl. \textbf{62} (1983), 63-72.

\bibitem{L2009}
 J. Lafontaine, \emph{A remark about static space times.} J. Geom. Phys. \textbf{59} (2009), 50-53.



\bibitem{L2015}
 B. Leandro, \emph{A note on critical point metrics of the total scalar curvature functional.} J. Math. Analysis and App. \textbf{424} (2015), 1544-1548.


\bibitem{MT2009}
 P.  Miao and L.-F. Tam, \emph{On the volume functional of compact manifolds with boundary with constant scalar curvature.} Calc. Var. PDE. \textbf{36} (2009), 141-171.

\bibitem{MT2011}
P. Miao  and L.-F. Tam， \emph{ Einstein and conformally flat critical metrics of the volume functional.} Trans. Amer. Math. Soc. \textbf{363} (2011), 2907-2937.







%
%










\bibitem{O1962}
M. Obata, \emph{Certain conditions for a Riemannian manifold to be isometric with a sphere.} J. Math. Soc. Japan \textbf{14(3)}  (1962), 333-340.


\bibitem{QY2013}
  J. Qing  and W. Yuan, \emph{A note on static spaces and related problems.} J. Geom. Phys. \textbf{74} (2013), 18-27.





\bibitem{QY2016}
J. Qing and W. Yuan, \emph{On scalar curvature rigidity of vacuum static spaces.} Math. Ann. \textbf{365} (2016), no. 3-4, 1257-1277. https://doi.org/10.1007/s00208-015-1302-0

\bibitem{SW2019}
W. Sheng and L. Wang, \emph{Critical metrics with cyclic parallel Ricci tensor for volume functional on manifolds with boundary.}
Geom. Dedic. \textbf{201} (2019), 243-251.




\bibitem{Y2015}
 W. Yuan, \emph{ The geometry of vacuum static spaces and deformations of scalar curvature.} Ph.D. thesis at
UC Santa Cruz (2015).




\end{thebibliography}
\end{document}